\documentclass[11pt,twoside]{amsart}

\usepackage{latexsym}
\usepackage{amsmath}
\usepackage{amssymb}
\usepackage{amsthm}
\usepackage{amscd}
\usepackage[all]{xy}
\usepackage{enumerate} 
\usepackage{amssymb} 
\usepackage{mathrsfs}

\def\ZZ{{\mathbb{Z}}}
\def\QQ{{\mathbb{Q}}}
\def\RR{{\mathbb{R}}}
\def\CC{{\mathbb{C}}}

\def\K{{\mathscr{K}}}

\def\PP{{\mathbb{P}}}
\def\U{{\mathscr{U}}}
\def\H{{\mathscr{H}}}

\usepackage{latexsym}

\newtheorem{them}{Theorem}[section]

\newtheorem{rem}[them]{Remark}

\newtheorem{cl}[them]{Claim}

\begin{document}

\title[Projective manifolds swept out by rational homogeneous varieties]{Classification of embedded projective manifolds swept out by rational homogeneous varieties of codimension one}
\author{Kiwamu Watanabe}
\date{January 9, 2011}

\address{Department of Mathematical Sciences  School of Science and Engineering Waseda University, 
4-1 Ohkubo 3-chome 
Shinjuku-ku 
Tokyo 169-8555 
Japan}
\email{kiwamu0219@fuji.waseda.jp}


\maketitle


\begin{abstract}
We give a classification of embedded smooth projective varieties swept out by rational homogeneous varieties whose Picard number and codimension are one.
\end{abstract}

\section{Introduction}

In the theory of polarized variety, it is a central problem to classify smooth projective varieties admitting special varieties $A$ as ample divisors.   
In the previous paper \cite{Wa}, we investigated this problem in the case where $A$ is a homogeneous variety. On the other hand, related to the classification problem of polarized varieties, several authors have been studied the structure of embedded projective varieties swept out by special varieties (see \cite{BI,MS,S,W2}).  Inspired by these results, in this short note, we give a classification of embedded smooth projective varieties swept out by rational homogeneous varieties whose Picard number and codimension are one. Our main result is

\begin{them}\label{MT} Let $X \subset \PP^N$ be a complex smooth projective variety of dimension $n \geq 3$ and $A$ an $(n-1)$-dimensional rational homogeneous variety with ${\rm Pic}(A) \cong \ZZ [\mathscr{O}_{A}(1)]$. Assume that $X$ satisfies either,
\begin{enumerate}
\item[\rm (a)] through a general point $x \in X$, there is a subvariety $Z_x \subset X$ such that $(Z_x, \mathscr{O}_{Z_x}(1))$ is isomorphic to $(A, \mathscr{O}_{A}(1))$, or 
\item[\rm (b)] there is a subvariety $Z \subset X$ such that $(Z, \mathscr{O}_{Z}(1))$ is isomorphic to $(A, \mathscr{O}_{A}(1))$ and the normal bundle $N_{Z/X}$ is nef. 
\end{enumerate}
Then $X$ is one of the following:
\begin{enumerate}
\item a projective space $\PP^n$,
\item a quadric hypersurface $Q^n$, 
\item the Grassmannian of lines $G(1,\PP^m)$, 
\item an $E_6$ variety $E_6(\omega_1)$, where $E_{6}(\omega_1) \subset \PP^{26}$ is  the projectivization of the highest weight vector orbit in the 27-dimensional irreducible representation of a simple algebraic group of Dynkin type $E_6$, \item $X$ admits an extremal contraction of a ray $\varphi  :X \rightarrow C$ to a smooth curve whose general fibers are projectively equivalent to $(A, \mathscr{O}_{A}(1))$.
\end{enumerate}
\end{them}

For each case $\rm (i)-(iv)$, the corresponding rational homogeneous variety $A$ is one of Theorem~\ref{ample}.  
We outline the proof of Theorem~\ref{MT}. A significant step is to show the existence of a covering family $\K$ of lines on $X$ induced from lines on rational homogeneous varieties of codimension one (Claim~\ref{family}). Then we see that the rationally connected fibration associated to $\K$ is an extremal contraction of the ray $\RR_{\geq 0}[\K]$. By applying the previous result \cite{Wa}, we obtain our theorem. In this paper, we work over the field of complex numbers.

\section{Preliminaries}

We denote a simple linear algebraic group of Dynkin type $G$ simply by $G$ and for a dominant integral weight $\omega$ of $G$, the minimal closed orbit of $G$ in $\PP(V_{\omega})$ by $G(\omega)$, where $V_{\omega}$ is the irreducible representation space of $G$ with highest weight $\omega$. 
For example, $E_6({\omega}_1)$ is the minimal closed orbit of an algebraic group of type $E_6$ in $\PP(V_{{\omega}_1})$, where ${\omega}_1$ is the first fundamental dominant weight in the standard notation of Bourbaki \cite{Bo}. Every rational homogeneous variety of Picard number one can be expressed by the form. Remark that a rational homogeneous variety $A$ is a Fano variety, i.e., the anti-canonical divisor of $A$ is ample.
If the Picard number of $A$ is one, we have ${\rm Pic}(A) \cong \ZZ[\mathscr{O}_{A}(1)]$, where $\mathscr{O}_{A}(1)$ is a very ample line bundle on $A$. 
We recall two results on rational homogeneous varieties.

\begin{them}[{\cite[Main Theorem]{HM2}, \cite[5.2]{HM1}}]\label{rigidity}
Let $A$ be a rational homogeneous variety of Picard number one. Let $\rho:X\to Z$ be a smooth proper morphism between two varieties. Suppose for some point $y$ on $Z$, the fiber $X_y$ is isomorphic to $A$. Then, for any point $z$ on $Z$, the fiber $X_z$ is isomorphic to $A$.
\end{them}

\begin{them}[\cite{Wa}]\label{ample} Let $X$ be a smooth projective variety and $A$ a rational homogeneous variety of Picard number one. If $A$ is an ample divisor on $X$, $(X, A)$ is isomorphic to $(\PP^n, \PP^{n-1})$, $(\PP^n, Q^{n-1})$, $(Q^n, Q^{n-1})$, $(G(2, \CC^{2l}), C_l(\omega_2))$ or $(E_6({\omega}_1), F_4({\omega}_4))$.
\end{them}

For a numerical polynomial $P(t) \in \QQ[t]$, we denote by ${\rm Hilb}_{P(t)}(X)$ the Hilbert scheme of $X$ relative to $P(t)$. More generally, 
for an $m$-tuple of numerical polynomials $\PP(t):=(P_1(t),\cdots,P_m(t))$, denote by ${\rm FH}_{\PP(t)}(X)$ the flag Hilbert scheme of $X$ relative to $\PP(t)$ (see \cite[Section~4.5]{Ser}).  For the Hilbert polynomial of a line $P_1(t)$,  an irreducible component of ${\rm Hilb}_{P_1(t)}(X)$ is called a {\it family of lines} on $X$.  
Let ${\rm Univ}(X)$ be the universal family of ${\rm Hilb}(X)$ with the associated morphisms $\pi :{\rm Univ}(X) \rightarrow {\rm Hilb}(X)$ and $\iota : {\rm Univ}(X) \rightarrow X$. 
For a subset $V$ of ${\rm Hilb}(X)$, $\iota ({\pi}^{-1}(V))$ is denoted by ${\rm Locus}(V) \subset X$. 
A {\it covering family of lines} $\K$ means an irreducible component of $F_1(X)$ satisfying ${\rm Locus}(\K)=X$. For a covering family of lines, we have the following fibration.

\begin{them}\cite{Ca, KMM}\label{RC} Let $X \subset \PP^N$ be a smooth projective variety and $\K$ a covering family of lines. Then there exists an open subset $X^0 \subset X$ and a proper morphism $\varphi : X^0 \rightarrow Y^0$ with connected fibers onto a normal variety, such that any two points on the fiber of $\varphi$ can be joined by a connected chain of finite $\K$-lines. 
\end{them}

We shall call the morphism $\varphi$ a {\it rationally connected fibration with respect to $\K$}.

\begin{them}\cite[Theorem~2]{BCD}\label{BCD} Under the condition and notation of Theorem~\ref{RC}, assume that $3 \geq \dim Y^0$. Then $\RR_{\geq 0}[\K]$ is extremal in the sense of Mori theory and the associated contraction
yields a rationally connected fibration with respect to $\K$. 
\end{them}

\begin{rem} \rm 
Remark that the original statements of Theorem~\ref{ample}, \ref{RC} and \ref{BCD} were dealt in more general situations. 
\end{rem}

\section{Proof of Theorem~\ref{MT}}

For a subset $V 	\subset X$, denote the closure by $\overline{V}$.  
Let $P_1(t)$, $P_2(t)$ be the Hilbert polynomials of a line, $(A, \mathscr{O}_{A}(1))$, respectively and set $\PP(t):=(P_1(t),P_2(t))$. We denote the natural projections by 
\begin{eqnarray}
p_i: {\rm FH}_{\PP(t)}(X) \rightarrow {\rm Hilb}_{P_i(t)}(X),~{\rm where}~i=1,2.\end{eqnarray} 
Let $\H$ be the open subscheme of ${\rm Hilb}_{P_2(t)}(X)$ parametrizing smooth subvarieties of $X$ with Hilbert polynomial $P_2(t)$. Now we work under the assumption that $X$ satisfies $\rm (a)$ or $\rm (b)$ in Theorem~\ref{MT}.

\begin{cl}\label{fam} In both cases $\rm (a)$ and $\rm (b)$, there exists a curve $C \subset \H$ which contains a point $o$ corresponding to a subvariety isomorphic to $(A,\mathscr{O}_A(1))$.
\end{cl}

\begin{proof} If the assumption $\rm (a)$ holds, there exists an irreducible component $\H_0$ of $\H$ which contains $o :=[Z_x]$ for some $x \in X$ and satisfies $\overline{{\rm Locus}(\H_0)}=X$. Then we can take a curve $C \subset \H_0$ which contains $o$. If the assumption $\rm (b)$ holds, we see that $h^1(N_{Z/X})=0$ and $h^0(N_{Z/X}) \geq 1$. Since there is no obstruction in the deformation of $Z$ in $X$, it turns out that $\H$ is smooth at $o:=[Z]$ and $\dim_{[Z]} \H \geq 1$. Then we can also take a curve $C \subset \H_0$ which contains $o$.
\end{proof}

From now on, we shall not use the assumptions $\rm (a)$ and $\rm (b)$ except the property proved in Claim~\ref{fam}. Note that $\overline{{\rm Locus}(C)}=X$. Denote by $\H_0$ an irreducible component of $\H$ which contains $C$. 
For the universal family $\pi: \U_0 \rightarrow \H_0$ and the normalization $\nu : \tilde{C} \rightarrow C \subset \H_0$, we denote $\tilde{C} \times_{\H_0} \U_0$ by $\U_{\tilde{C}}$ and a natural projection by $\tilde{\pi}: \U_{\tilde{C}} \rightarrow \tilde{C}$. Let $(\U_{\tilde{C}})_{\rm red}$ be the reduced scheme associated to $\U_{\tilde{C}}$ and $\Pi : (\U_{\tilde{C}})_{\rm red} \rightarrow \tilde{C}$ the composition of $\tilde{\pi}$ and $(\U_{\tilde{C}})_{\rm red} \rightarrow \U_{\tilde{C}}$. Then we have the following diagram:

\[\xymatrix{
(\U_{\tilde{C}})_{\rm red} \ar[r] \ar[dr]_{\Pi} &\U_{\tilde{C}} \ar[r] \ar[d]_{\tilde{\pi}} & \U_0 \ar[d]_\pi \\
& \tilde{C} \ar[r]_{\nu} &  \H_0 \\
} \]

Now we have an isomorphism between scheme theoretic fibers $\tilde{\pi}^{-1}(p) \cong \pi^{-1}(\nu (p))$ for any closed point $p \in \tilde{C}$. In particular,  $\tilde{\pi}^{-1}(p)$ is a smooth projective variety and $\tilde{\pi}^{-1}(\tilde{o}) \cong A$ for a point $\tilde{o} \in \tilde{C}$ corresponding to $o \in C$. Moreover a natural morphism $\Pi^{-1}(p) \rightarrow \tilde{\pi}^{-1}(p)$ is a homeomorphic closed immersion for any closed point $p \in \tilde{C}$. Since $\tilde{\pi}^{-1}(p)$ is reduced, we see that $\Pi^{-1}(p) \cong \tilde{\pi}^{-1}(p)$. Thus it concludes that $\Pi$ is a proper flat morphism whose fibers on closed points are smooth projective varieties, that is, a proper smooth morphism. Because $\Pi$ admits a central fiber $\tilde{\Pi}^{-1}(\tilde{o}) \cong A$, it follows that every fiber $\tilde{\Pi}^{-1}(\tilde{p})$ is isomorphic to $A$ from Theorem~\ref{rigidity}. Hence it turns out that every fiber of $\pi$ over a closed point in $C$ is isomorphic to $A$. 
Let consider a constructible subset $p_1(p_2^{-1}(C)) \subset {\rm Hilb}_{P_1(t)}(X)$. Since $C$ parametrizes subvarieties isomorphic to $(A, \mathscr{O}_A(1))$ which is covered by lines, we see that $\overline{{\rm Locus}(p_1(p_2^{-1}(C)))}=X$. 

\begin{cl}\label{family} There exists a covering family of lines $\K$ on $X$ satisfying the following property: 
Through a general point $x \in X$, there is a subvariety $S_x \subset X$ so that $(S_x, \mathscr{O}_{S_x}(1)) \cong (A, \mathscr{O}_{A}(1))$ and any line lying in $S_x$ is a member of $\K$.
\end{cl}

\begin{proof} Take an irreducible component $\K^0$ of $p_1(p_2^{-1}(C))$ such that $\overline{{\rm Locus}(\K^0)}=X$. 
Through a general point $x$ on $X$, there is a line $[l_x] $ in $\K^0$ which is not contained in any irreducible component of $p_1(p_2^{-1}(C))$ except $\K^0$. Furthermore there is also a subvariety $[S_x]$ in $C$ containing $l_x$. Because $p_1(p_2^{-1}([S_x]))$ is the Hilbert scheme of lines on $S_x$, it is irreducible (see \cite[Theorem~4.3]{LM} and \cite[Theorem~1]{E}). Therefore $p_1(p_2^{-1}([S_x]))$ is contained in an irreducible component of $p_1(p_2^{-1}(C))$. Since $p_1(p_2^{-1}([S_x]))$ contains $[l_x]$, this implies that $p_1(p_2^{-1}([S_x]))$ is contained in $\K^0$. 
Thus we put $\K$ as an irreducible component of ${\rm Hilb}_{P_1(t)}(X)$ containing $\K^0$.
\end{proof}

Two points on $S_x \cong A$ can be joined by a connected chain of lines in $\K$. It implies that the relative dimension of the rationally connected fibration $\varphi: X \cdots \rightarrow Y$ with respect to $\K$ is at least $n-1$. According to Theorem~\ref{BCD}, $\RR_{\geq 0}[\K]$ spans an extremal ray of $NE(X)$ and $\varphi$ is its extremal contraction. In particular, $\varphi$ is a morphism which contracts $S_x$ to a point. If $\dim Y=0$, we see that the Picard number of $X$ is one. It implies that $S_x$ is a very ample divisor on $X$. From Theorem~\ref{ample}, $X$ is $\PP^n$, $Q^n$, $G(1,\PP^m)$ or $E_6(\omega_1)$. If $\dim Y=1$, then $Y$ is a smooth curve $C$ and a general fiber of $\varphi$ coincides with $S_x$. Therefore $\varphi$ is a $A$-fibration on a smooth curve $C$. Hence Theorem~\ref{MT} holds. \\

{\bf Acknowledgements}
The author would like to thank Professor Hajime Kaji for a valuable seminar. The author is supported by Research Fellowships of the Japan Society for the Promotion of Science for Young Scientists. 
\begin{footnotesize}

\end{footnotesize}


\begin{thebibliography}{7}

\bibitem{BI} M. C. Beltrametti, P. Ionescu, {\it On manifolds swept out by high dimensional quadrics}, Math. Z. 260 (2008), no. 1, 229-234.

\bibitem{BCD} L. Bonavero, C. Casagrande, S. Druel, {\it On covering and quasi-unsplit families of curves}, J. Eur. Math. Soc. (JEMS) 9 (2007), no. 1, 45-57.

\bibitem{Bo} N. Bourbaki: {\it $\acute{E}$l$\acute{e}$ments de Math$\acute{e}$matique, Groupes et alg$\grave{e}$bres de Lie}, Chapters 4,5, and 6, (Hermann, Paris, 1968).

\bibitem{Ca} F. Campana, {\it Connexite rationelle des varieties de Fano}, Ann. Sci. Ec. Norm. Sup. 25 (1992), 539-545

\bibitem{E} S. Elisabetta, {\it Lines in $G/P$}, Math. Z. 242 no. 2 (2002), 227-240.

\bibitem{HM1} J. M. Hwang and N. Mok, {\it Rigidity of irreducible Hermitian symmetric spaces of the compact type under K$\ddot{a}$hler deformation}, Invent. Math. 131 no. 2 (1998), 393-418.

\bibitem{HM2} J. M. Hwang and N. Mok, {\it Prolongations of infinitesimal linear automorphisms of projective varieties and rigidity of rational homogeneous spaces of Picard number 1 under K$\ddot{a}$hler deformation}, Invent. Math. 160 (2005), no. 3, 591-645.

\bibitem{Ko} J. Koll\'ar, {\it Rational curves on algebraic varieties}, Ergebnisse der Mathematik und ihrer Grenzgebiete (3), vol. 32 (Springer, Berlin, 1996).

\bibitem{KMM} J. Koll\'ar, Y. Miyaoka and S. Mori, {\it Rational connectedness and boundedness of Fano manifolds}, J. Differential Geom. 36 no. 3 (1992), 765-779.

\bibitem{LM} J. M. Landsberg and L. Manivel, {\it On the projective geometry of rational homogeneous varieties}, Comment. Math. Helv. 78 no. 1 (2003), 65-100.

\bibitem{MS} R. Mu$\rm \tilde{n}$oz and L. E. Sol$\rm {\acute{a}}$, {\it Varieties swept out by Grassmannians of lines}, Contemp. Math. 496 (2009), 303-316.

\bibitem{S} E. Sato, {\it Projective manifolds swept out by large-dimensional linear spaces}, Tohoku Math. J. (2) 49 (1997), no. 3, 299-321.


\bibitem{Ser} E. Sernesi, Deformations of algebraic schemes. Grundlehren der Mathematischen Wissenschaften [Fundamental Principles of Mathematical Sciences], 334. Springer-Verlag, Berlin, 2006. xii+339 pp. ISBN: 978-3-540-30608-5; 3-540-30608-0.

\bibitem{Wa} K. Watanabe, {\it Classification of polarized manifolds admitting homogeneous varieties as ample divisors}, Math. Ann. 342 (2008), no. 3, 557-563.

\bibitem{W2} K. Watanabe, {\it On projective manifolds swept out by high dimensional cubic varieties}, preprint (2010). arXiv:1010.2300v1.








\end{thebibliography}
\end{document}